\pdfoutput=1
\documentclass[12pt,oneside]{amsart}
\usepackage{placeins, comment}
\usepackage[margin=1in]{geometry}
\usepackage{mathabx}
\usepackage{latexsym,epsfig,enumitem,multicol,wasysym,amsmath,amsthm,amssymb,amsfonts,amscd,ulem,soul,multirow,rotating,pdflscape,float,dsfont}
\setlist{itemsep=.06125in}

\usepackage{graphicx}
\usepackage[colorlinks,citecolor=red,pagebackref,hypertexnames=false]{hyperref}
\hypersetup{%
  colorlinks = true,
 citecolor=red,
  linkcolor  = red
}
\restylefloat{table}
\numberwithin{equation}{section}
\theoremstyle{plain}
\newtheorem{theorem}{Theorem}[section]
\newtheorem{lemma}[theorem]{Lemma}
\newtheorem{corollary}[theorem]{Corollary}
\newtheorem{proposition}[theorem]{Proposition}

\theoremstyle{definition}
\newtheorem{definition}[theorem]{Definition}

\theoremstyle{remark}
\newtheorem{remark}[theorem]{Remark}

\date{\today}      
\author{A. Iosevich and A. Yavicoli}
\address{Department of Mathematics, University of Rochester, Rochester, NY, USA}
\email{iosevich@gmail.com}
\address{Department of Mathematics, University of British Columbia, Vancouver, BC, Canada}
\email{yavicoli@math.ubc.ca} 

\thanks{A.~I. was supported in part by the National Science Foundation under NSF DMS - 2154232.}

\thanks{A.~Y. was supported in part by the Natural Sciences and Engineering Research Council of Canada, NSERC (GR030571 and GR030540).}

\title[Falconer lattice sets and the Erdos similarity problem]{Falconer lattice sets and the Erdos similarity problem}

\begin{document} 

\begin{abstract}
We show that a family of extremely thin sets satisfy the Erd\H{o}s similarity conjecture. These examples lie outside the range covered by recent work of Shmerkin and Yavicoli \cite{ShmerkinYavicoli2025}. As we shall see, they have small logarithmic dimension. They do not contain affine copies of slowly decaying sequences, so the result does not follow from earlier work of Falconer and Eigen \cite{Falconer1984,Eigen}. On the other hand, they do contain sequences of rapid decay, for which the conjecture is still open in general. Our argument is based on Falconer lattice sets and a theorem of Bourgain \cite{Bourgain2003}.

More precisely, for a monotone increasing function $\phi$ tending to infinity, we consider the intersection of neighborhoods of the lattices
$$
\frac{1}{q_i}\big([0,q_i]\cap {\mathbb Z}\big),
\qquad q_1=2, \qquad q_{i+1}=q_i^{M_i},
$$
with neighborhood radii $q_i^{-\phi(q_i)}$. We show that if
$\phi(q_i)<M_i-1$ for all sufficiently large $i$, then the resulting intersection set $E$ contains an infinite set $A$ such that
$$
A+A+A\subset E.
$$
By Bourgain's theorem \cite{Bourgain2003}, it follows that $E$ satisfies the conclusion of the Erd\H{o}s similarity conjecture.

We also prove a structural dichotomy. If $\phi(q_i)>M_i$
for all sufficiently large $i$, then the intersection $E$ is finite. If $\phi(q_i)<M_i-1$ for all sufficiently large $i$, then $E$ is uncountable and contains an infinite set $A$ with
$$
A+A+A\subset E.
$$
In addition, we show that these sets avoid the previously studied slowly decaying sequence regime $\frac{a_{n+1}}{a_n}\to 1,$ so the mechanism forcing the Erd\H{o}s similarity property here comes from additive branching generated by the nested lattice structure rather than from slow decay of sequences.

We also show that Bourgain's result on sumsets is not enough to prove the Erd\H{o}s similarity conjecture for Cantor sets.

In addition, we show another family of Cantor sets with logarithmic dimension $0$ that satisfy the Erd\H{o}s similarity conjecture (these are not built using Falconer's lattice sets).
\end{abstract}

\subjclass[2020]{Primary 28A78; Secondary 11B30, 28A80, 11K60, 05D99.}

\keywords{Falconer lattice sets, Erd\H{o}s similarity problem, geometric combinatorics, additive structure, sumsets, triple sum sets, nested lattices, Hausdorff dimension, logarithmic dimension, thin sets, Cantor-type constructions, Bourgain theorem, additive branching, Diophantine approximation, fractal geometry.}

\maketitle

\tableofcontents

\section{Introduction}
\label{section:introduction}

The Erd\H{o}s similarity problem is a classical question in geometric combinatorics. It asks whether for every infinite set $C$ there is a measurable subset of the real line of positive measure  so that the set does not contain an affine copy of $C$.  This problem was raised by Erd\H{o}s in the 1970s. See \cite{Erdos1974}.

Although the conjecture is wide open, there are some results that are worth mentioning. First of all, it is believed that the conjecture is likely to be true, since Kolountzakis \cite{OLDKol} showed that the conjecture is almost surely true, in the following sense: For every infinite set $C$, there is a set $E$ of positive measure such that
\[
(x + tC) \subseteq E
\]
fails for almost all (Lebesgue) pairs $(x,t)$.

Recent work of Shmerkin and Yavicoli \cite{ShmerkinYavicoli2025} proves the Erd\H{o}s similarity conjecture for sets $C$ of large logarithmic dimension.
Their result is based on sumsets in the discrete setting and an inductive argument.

In the present paper we prove the conclusion of the Erd\H{o}s similarity conjecture for a family of sets arising from a modified Falconer lattice construction. Our approach is to produce a strong additive branching across scales, which forces the existence of an infinite set $A\subset E$ with $A+A+A\subset E$. An important breakthrough in the study of this problem was obtained by Bourgain. He proved that if a set $E$ contains a triple sum
$$
A+A+A
$$
for some infinite set $A$, then $E$ satisfies the conclusion of the Erd\H{o}s similarity conjecture. See \cite{Bourgain2003}. As we shall see later, the examples considered here also have small logarithmic dimension.

More recently, using a probabilistic method, Kolountzakis \cite{Kolountzakis2023} showed that certain thin symmetric Cantor sets satisfy the conjecture. His work also highlights a Bourgain-type additive mechanism for non-universality in the symmetric Cantor setting. The construction in the present paper is different in nature: it is based on Falconer lattice sets and the additive branching forced by nested lattices across scales. In particular, the examples considered here arise from a different geometric mechanism and, as we discuss below, have small logarithmic dimension. Kolountzakis's construction uses symmetry and recursive Cantor dissection, whereas our argument is driven by nested lattice branching and applies to a different class of thin sets.

Earlier work of Falconer \cite{Falconer1984} studies the Erd\H{o}s similarity problem for sequences that decay very slowly. In particular, sequences $\{a_n\}$ satisfying
$$
\frac{a_{n+1}}{a_n}\to 1
$$
belong to a regime that has been extensively studied in the literature. 

Our construction is inspired by a classical example due to Falconer. Falconer introduced a lattice-based construction in which one considers neighborhoods of dilates of the integer lattice and showed that the resulting intersection may have a prescribed Hausdorff dimension. Sets constructed in this way played an important role in the study of sharpness phenomena in geometric measure theory, including the Falconer distance problem.

In the one-dimensional setting Falconer's construction proceeds as follows. Let
$$
q_1=2, \qquad q_{i+1}=q_i^{M_i},
$$
and let $E_{i,s}$ denote the $q_i^{-1/s}$-neighborhood of
$$
\frac{1}{q_i}\big([0,q_i]\cap \mathbb Z\big).
$$
Falconer showed that the intersection
$$
E_s=\bigcap_{i=1}^{\infty} E_{i,s}
$$
has Hausdorff dimension $s$. See \cite{Falconer1990,Falconer2003}.

The essential feature of Falconer's construction for our purposes is that it is generated by nested lattices. As a result, additive relations between lattice points persist across scales.

In this paper we modify Falconer's construction by allowing the neighborhood radius to vary according to a function $\phi$. Specifically, we consider neighborhoods of the form
$$
q_i^{-\phi(q_i)}
$$
around the lattice
$$
\frac{1}{q_i}\big([0,q_i]\cap \mathbb Z\big).
$$

We analyze the structure of the resulting intersection sets. In particular we show that, in a natural regime for the function $\phi$, these sets contain a triple sum of an infinite set. Bourgain's theorem from \cite{Bourgain2003} then implies that these sets satisfy the conclusion of the Erd\H{o}s similarity conjecture.

A key feature of the construction is a threshold phenomenon governed by the comparison between the lattice spacing $1/q_{i+1}$ and the neighborhood radius $q_i^{-\phi(q_i)}$. This leads to two distinct regimes: one in which the intersection collapses to a finite set, and another in which it contains rich additive structure.

It is important to emphasize that the mechanism underlying Theorem \ref{thm:main} is fundamentally different from previously studied regimes. The result of Shmerkin and Yavicoli \cite{ShmerkinYavicoli2025} applies to sets of logarithmic Hausdorff dimension larger than $1$ or sets of logarithmic Packing dimension larger than $2$, whereas the examples constructed here can have logarithmic Packing dimension strictly less than $2$, and under an additional growth condition also have logarithmic Hausdorff dimension strictly less than $1$, so in particular can fall outside that regime.

We recall the definition of logarithmic Hausdorff and packing dimensions from \cite{ShmerkinYavicoli2025}, starting from the general notion of Hausdorff and packing measures associated with a dimension function.

The space of dimension or gauge functions is defined as
\[
\Phi := \left\{
\phi : \mathbb{R}_{\geq 0} \to \mathbb{R}_{\geq 0} :
\phi(t) > 0 \text{ if } t > 0,\ \phi(0) = 0,\ \text{increasing and right-continuous}
\right\}.
\]
This set is partially ordered, with the order defined by
\[
\phi_2 \prec \phi_1 \quad \text{if} \quad \lim_{x \to 0^+} \frac{\phi_1(x)}{\phi_2(x)} = 0.
\]

The outer Hausdorff measure associated with $\phi \in \Phi$ is
\[
\mathcal{H}^\phi(E) :=
\lim_{\delta \to 0}
\inf \left\{
\sum_i \phi(\operatorname{diam}(U_i)) :
\operatorname{diam}(U_i) < \delta,\ 
E \subset \bigcup_i U_i
\right\}.
\]

A packing of a set $E$ is a collection of disjoint balls with centres in $E$. A $\delta$-packing
of $E$ is a packing in which all radii are less than $\delta$.

Let $\phi \in \Phi$. The packing pre-measure associated with $\phi$ is defined as
\[
\mathcal{P}^\phi_0(E) :=
\lim_{\delta \to 0}
\sup \left\{
\sum_i \phi(r_i) :
B(x_i, r_i) \text{ is a } \delta\text{-packing of } E
\right\}.
\]

The outer packing measure associated with $\phi$ is defined as
\[
\mathcal{P}^\phi(E) :=
\inf \left\{
\sum_i \mathcal{P}^\phi_0(E_i) :
E \subset \bigcup_i E_i
\right\}.
\]

\begin{definition}
The logarithmic Hausdorff (resp.\ packing) dimension of a set $A \subset \mathbb{R}^d$,
denoted $\dim^{\log}_{H}(A)$ (resp.\ $\dim^{\log}_{P}(A)$), is defined as the infimum of all $s > 0$ such that
\[
\mathcal{H}^{h_s}(A) = 0
\quad \text{(resp.\ } \mathcal{P}^{h_s}(A) = 0\text{)},
\]
where
\[
h_s(r) = (\log(1/r))^{-s}
\]
for sufficiently small $r>0$.
\end{definition}

On the other hand, earlier work of Falconer \cite{Falconer1984} and Eigen \cite{Eigen} applies to sets containing slowly decaying sequences satisfying
$$
\frac{a_{n+1}}{a_n}\to 1.
$$

Lemma \ref{lemma:noslowsequences} below shows that the sets constructed in the present paper do not contain affine copies of such slowly decaying sequences. Thus, the conclusion of the Erd\H{o}s similarity conjecture in our setting does not follow from those results.

Instead, the key mechanism here is a branching additive structure generated by nested lattices across scales, which produces infinite sets $A\subset E$ with $A+A+A\subset E$. This provides a new route to the Erd\H{o}s similarity property that is independent of both the slow decay and positive logarithmic dimension frameworks.

\begin{theorem}\label{thm:main}
Let
$$
q_1=2, \qquad q_{i+1}=q_i^{M_i},
$$
where $(M_i)_{i=1}^{\infty}$ is an increasing sequence of positive integers such that
$$
M_i\to\infty.
$$

Let $\phi:[2,\infty)\to(0,\infty)$ be monotone increasing with $\phi(t)\to\infty$.

For each $i$, define
$$
G_i=\frac{1}{q_i}\big([0,q_i]\cap {\mathbb Z}\big),
$$

$$
r_i=q_i^{-\phi(q_i)},
$$

and
$$
E_i=\{x\in [0,1]: \operatorname{dist}(x,G_i)\le r_i\}.
$$

Set
$$
E=\bigcap_{i=1}^{\infty} E_i.
$$

If
$$
\phi(q_i)<M_i-1
$$
for all sufficiently large $i$, then there exists an infinite set
$$
A\subset E
$$
such that
$$
A+A+A\subset E.
$$

Consequently, by Bourgain's theorem \cite{Bourgain2003}, the set $E$ satisfies the conclusion of the Erd\H{o}s similarity conjecture.
\end{theorem}

\vskip.125in 

The Cantor-type sets constructed in the present paper arise from a nested lattice structure that forces strong additive branching across scales. This branching produces infinite sets $A \subset E$ with
$$
A+A+A \subset E,
$$
which by Bourgain's theorem implies that $E$ satisfies the conclusion of the Erd\H{o}s similarity conjecture.

The proof is given in Section \ref{section:modifiedfalconer}. We analyze the structure of the modified Falconer construction and show that the behavior of the intersection set $E$ exhibits a threshold phenomenon. When the neighborhood radius is sufficiently large compared to the next lattice spacing, the construction produces a highly branching structure and the intersection contains a triple sum of an infinite set. When the neighborhoods are too small, the intersection collapses to a finite set.

\begin{proposition}[Structural dichotomy]\label{prop:dichotomy}
Let
$$
q_1=2, \qquad q_{i+1}=q_i^{M_i},
$$
and define
$$
G_i=\frac{1}{q_i}\big([0,q_i]\cap {\mathbb Z}\big), \qquad
r_i=q_i^{-\phi(q_i)},
$$
$$
E_i=\{x\in[0,1]:\operatorname{dist}(x,G_i)\le r_i\},
\qquad
E=\bigcap_{i=1}^{\infty}E_i.
$$

If
$$
\phi(q_i)>M_i
$$
for all sufficiently large $i$, then $E$ is finite.

If
$$
\phi(q_i)<M_i-1
$$
for all sufficiently large $i$, then $E$ is uncountable. Moreover, there exists an infinite sequence
$$
a_1,a_2,a_3,\dots
$$
such that
$$
\{a_1,a_2,a_3,\dots\}\subset E
$$
and
$$
\{a_1,a_2,a_3,\dots\}+\{a_1,a_2,a_3,\dots\}+\{a_1,a_2,a_3,\dots\}\subset E.
$$
\end{proposition}

These examples appear to be substantially thinner than the large logarithmic dimension regime considered in \cite{ShmerkinYavicoli2025}, although we do not pursue the precise logarithmic dimension here.

The proof of this structural dichotomy is given in Section \ref{section:modifiedfalconer}.

\vskip.125in

In addition to the main construction and its consequences, we address two further
questions related to the scope of the method. In Section 3 we show that Bourgain's
sumset theorem alone is not sufficient to prove the Erd\H{o}s similarity conjecture
for general Cantor-type sets, thereby clarifying the limitations of this approach.
In Section 4 we construct another family of Cantor sets of logarithmic Hausdorff
dimension zero that nevertheless satisfy the conclusion of the Erd\H{o}s similarity
conjecture. These examples arise from a different mechanism than the lattice-based
construction considered above.

\vskip.125in 

\section{Modified Falconer construction}
\label{section:modifiedfalconer}

Let
$$
q_1=2, \qquad q_{i+1}=q_i^{M_i}.
$$

For each $i$ define
$$
G_i=\frac{1}{q_i}\big([0,q_i]\cap {\mathbb Z}\big).
$$

Let $\phi:[2,\infty)\to(0,\infty)$ be a monotone increasing function with $\phi(t)\to\infty$, and define
$$
r_i=q_i^{-\phi(q_i)}.
$$

We consider the sets
$$
E_i=\{x\in[0,1]:\operatorname{dist}(x,G_i)\le r_i\}
$$
and their intersection
$$
E=\bigcap_{i=1}^{\infty}E_i.
$$

The nested lattice structure $G_i\subset G_{i+1}$ plays a crucial role in the analysis of the set $E$.

\begin{lemma}[Branching of the lattice structure]
Let
$$
q_1=2, \qquad q_{i+1}=q_i^{M_i},
$$
where $(M_i)_{i=1}^{\infty}$ is an increasing sequence of positive integers with
$$
M_i\to\infty.
$$

Let
$$
G_i=\frac{1}{q_i}\big([0,q_i]\cap {\mathbb Z}\big),
$$

$$
r_i=q_i^{-\phi(q_i)},
$$

and
$$
E_i=\{x\in[0,1]:\operatorname{dist}(x,G_i)\le r_i\}.
$$

Let
$$
E=\bigcap_{i=1}^{\infty}E_i.
$$

Assume that
$$
\phi(q_i)<M_i-1
$$
for all sufficiently large $i$.

Then for every sufficiently large $i$ and every $g\in G_i$ the interval
$$
[g-r_i,g+r_i]
$$
contains at least
$$
2q_i^{\,M_i-\phi(q_i)}-1
$$
points of the finer lattice $G_{i+1}$.

In particular the number of descendants of $G_i$ in the next lattice that remain inside $E_i$ is at least $q_i$ (which goes to $\infty$).
\end{lemma}

\begin{proof}
The points of $G_{i+1}$ are spaced by
$$
\frac{1}{q_{i+1}}.
$$

The interval $[g-r_i,g+r_i]$ has length $2r_i$.

Therefore the number of points of $G_{i+1}$ inside this interval is at least
$$
2r_i q_{i+1}-1.
$$

Since
$$
q_{i+1}=q_i^{M_i}
$$
and
$$
r_i=q_i^{-\phi(q_i)},
$$
we obtain
$$
r_i q_{i+1}=q_i^{\,M_i-\phi(q_i)}.
$$

Hence
$$
2r_i q_{i+1}-1=2q_i^{\,M_i-\phi(q_i)}-1.
$$

Under the hypothesis
$$
\phi(q_i)<M_i-1,
$$
this quantity tends to infinity with $i$.

Therefore every lattice point $g\in G_i$ has arbitrarily many points of $G_{i+1}$ inside the interval $[g-r_i,g+r_i]$.

These descendants all belong to $E_i$ by construction.
\end{proof}

\begin{lemma}[The intersection is uncountable]\label{lemma:largeintersection}
Let
$$
q_1=2, \qquad q_{i+1}=q_i^{M_i},
$$
and for each $i$ let
$$
G_i=\frac{1}{q_i}\big([0,q_i]\cap \mathbb Z\big).
$$
Let $\phi:[2,\infty)\to(0,\infty)$ be monotone increasing with $\phi(t)\to\infty$, and define
$$
r_i=q_i^{-\phi(q_i)},
$$
$$
E_i=\{x\in[0,1]:\operatorname{dist}(x,G_i)\le r_i\}.
$$
Assume that
$$
\phi(q_i)<M_i-1
$$
for all sufficiently large $i$. Then
$$
E=\bigcap_{i=1}^{\infty}E_i
$$
is uncountable.
\end{lemma}

\begin{proof}
Since $q_{i+1}=q_i^{M_i}$, we have $q_i\mid q_{i+1}$ and therefore
$$
G_i\subset G_{i+1}.
$$

Because
$$
r_i q_{i+1}=q_i^{M_i-\phi(q_i)},
$$
and
$$
\phi(q_i)<M_i-1
$$
for all sufficiently large $i$, it follows that
$$
r_i q_{i+1}\to\infty.
$$

Also, since $\phi(q_i)\to\infty$, we have
$$
q_i r_i=q_i^{1-\phi(q_i)}\to 0.
$$
Since $q_{i+1}\to\infty$ and $\phi(q_{i+1})\to\infty$, it follows in the same way that
$$
q_{i+1}r_{i+1}=q_{i+1}^{1-\phi(q_{i+1})}\to 0.
$$

Since $r_i q_{i+1}=q_i^{M_i-\phi(q_i)}\to\infty$ and $r_{i+1}q_{i+1}=q_{i+1}^{1-\phi(q_{i+1})}\to 0$, it follows that
$$
(r_i-r_{i+1})q_{i+1}\to\infty.
$$

Choose $i_0$ so large that for every $i\ge i_0$ the following two properties hold:
$$
(r_i-r_{i+1})q_{i+1}>3
$$
and
$$
2r_{i+1}<\frac{1}{q_{i+1}}.
$$

We now construct a binary tree of nested intervals.

At level $i_0$, set
$$
J_{i_0}=[0,r_{i_0}].
$$
Since $0\in G_{i_0}$, this is a closed subinterval of $[0,1]$.

Suppose inductively that at level $n\ge i_0$ we have constructed a closed interval
$$
J_n=[x_n,x_n+r_n]
$$
with
$$
x_n\in G_n.
$$
Because
$$
(r_n-r_{n+1})q_{n+1}>3,
$$
the interval
$$
[x_n,x_n+r_n-r_{n+1}]
$$
contains at least two distinct points of the lattice $G_{n+1}$. Choose two such points and call them
$$
x_{n+1}^{(0)}, \qquad x_{n+1}^{(1)}.
$$
Define the corresponding child intervals
$$
J_{n+1}^{(0)}=[x_{n+1}^{(0)},x_{n+1}^{(0)}+r_{n+1}],
$$
$$
J_{n+1}^{(1)}=[x_{n+1}^{(1)},x_{n+1}^{(1)}+r_{n+1}].
$$
By construction,
$$
J_{n+1}^{(0)}\subset J_n
\qquad \text{and} \qquad
J_{n+1}^{(1)}\subset J_n.
$$

Moreover, since the two centers are distinct points of $G_{n+1}$, they are separated by at least $1/q_{n+1}$. Because
$$
2r_{n+1}<\frac{1}{q_{n+1}},
$$
the intervals
$$
J_{n+1}^{(0)} \quad \text{and} \quad J_{n+1}^{(1)}
$$
are disjoint.

Thus every vertex has two disjoint children.

Now let
$$
\omega=(\omega_{i_0+1},\omega_{i_0+2},\omega_{i_0+3},\dots),
\qquad \omega_j\in\{0,1\},
$$
be an infinite binary sequence. Following the corresponding choices gives a nested sequence of nonempty closed intervals
$$
J_{i_0}\supset J_{i_0+1}^{(\omega_{i_0+1})}\supset J_{i_0+2}^{(\omega_{i_0+2})}\supset \cdots.
$$
By compactness,
$$
\bigcap_{n=i_0}^{\infty} J_n^{(\omega)}\neq\emptyset.
$$
Choose
$$
x_\omega\in \bigcap_{n=i_0}^{\infty} J_n^{(\omega)}.
$$

We claim that
$$
x_\omega\in E_i
$$
for every $i$.

If $i\ge i_0$, then
$$
x_\omega\in J_i^{(\omega)}=[x_i,x_i+r_i]
$$
for some point $x_i\in G_i$, so
$$
\operatorname{dist}(x_\omega,G_i)\le r_i.
$$
Hence
$$
x_\omega\in E_i.
$$

If $i<i_0$, then
$$
x_\omega\in J_{i_0}=[0,r_{i_0}],
$$
and since the sequence $(r_i)$ is decreasing,
$$
r_{i_0}\le r_i.
$$
Because $0\in G_i$, we get
$$
\operatorname{dist}(x_\omega,G_i)\le x_\omega\le r_{i_0}\le r_i,
$$
so again
$$
x_\omega\in E_i.
$$

Therefore
$$
x_\omega\in \bigcap_{i=1}^{\infty}E_i=E.
$$

Finally, if
$$
\omega\neq \omega',
$$
let $m$ be the first level at which they differ. Then the corresponding level-$m$ intervals are disjoint, so
$$
x_\omega\neq x_{\omega'}.
$$
Hence the map
$$
\omega\mapsto x_\omega
$$
is injective from the set of infinite binary sequences into $E$.

Since the set of infinite binary sequences is uncountable, the set $E$ is uncountable.
\end{proof}

The following Lemma ensures that the set does not contain homothetic copies of slowly decreasing to $0$ sequences (as in Eigen and Falconer setting).

\begin{lemma}[No slowly decreasing sequences near any point of the intersection]
\label{lemma:noslowsequences}
Let
$$
q_1=2, \qquad q_{i+1}=q_i^{M_i},
$$
$$
G_i=\frac{1}{q_i}\big([0,q_i]\cap {\mathbb Z}\big),
$$
$$
r_i=q_i^{-\phi(q_i)},
$$
and
$$
E_i=\{x\in [0,1]: \operatorname{dist}(x,G_i)\le r_i\}.
$$
Assume that
$$
\phi(q_i)\to \infty.
$$
Set
$$
E=\bigcap_{i=1}^{\infty} E_i.
$$
Then there do not exist a point $x_0\in E$ and a sequence of distinct points $\{a_n\}\subset E$ such that
$$
a_n\to x_0
\qquad \text{and} \qquad
\frac{|a_{n+1}-x_0|}{|a_n-x_0|}\to 1.
$$
\end{lemma}

\begin{proof}
Suppose for contradiction that there exist $x_0\in E$ and a sequence of distinct points $\{a_n\}\subset E$ such that
$$
a_n\to x_0
\qquad \text{and} \qquad
\frac{|a_{n+1}-x_0|}{|a_n-x_0|}\to 1.
$$

The key point is that the argument localizes around the accumulation point $x_0$. For each sufficiently large scale $i$, the set $E$ near $x_0$ lies inside a single interval of length comparable to $r_i$ around a lattice point $g_i\in G_i$. This forces a rapid drop in the successive distances to $x_0$, which contradicts the assumption that the ratios tend to $1$.

Since $a_n\to x_0$ and the points are distinct, after passing to a subsequence (which does not affect the condition $|a_{n+1}-x_0|/|a_n-x_0|\to 1$) we may assume that
$$
0<|a_{n+1}-x_0|<|a_n-x_0|
$$
for every $n$.

Since $\phi(q_i)\to\infty$, we have
$$
q_i r_i=q_i^{1-\phi(q_i)}\to 0.
$$
Hence there exists $i_0$ such that for all $i\ge i_0$,
$$
r_i<\frac{1}{4q_i}.
$$

Fix $i\ge i_0$. Since $x_0\in E\subset E_i$, there exists a point $g_i\in G_i$ such that
$$
|x_0-g_i|\le r_i.
$$

We claim that
$$
E\cap \left[g_i-\frac{1}{2q_i},g_i+\frac{1}{2q_i}\right]
\subset
[g_i-r_i,g_i+r_i].
$$

To prove this, let
$$
x\in E\cap \left[g_i-\frac{1}{2q_i},g_i+\frac{1}{2q_i}\right].
$$
Since $x\in E\subset E_i$, there exists $h_i\in G_i$ such that
$$
|x-h_i|\le r_i.
$$
Now the points of $G_i$ are spaced by $1/q_i$. Because
$$
r_i<\frac{1}{4q_i},
$$
the interval
$$
\left[g_i-\frac{1}{2q_i},g_i+\frac{1}{2q_i}\right]
$$
can meet the $r_i$-neighborhood of only the single grid point $g_i$. Indeed, if $h_i\ne g_i$, then
$$
|h_i-g_i|\ge \frac{1}{q_i},
$$
so for every $x$ with $|x-g_i|\le \frac{1}{2q_i}$ we get
$$
|x-h_i|\ge |h_i-g_i|-|x-g_i|
\ge \frac{1}{q_i}-\frac{1}{2q_i}
=
\frac{1}{2q_i}
>
r_i,
$$
which is impossible. Therefore $h_i=g_i$, and hence
$$
|x-g_i|\le r_i.
$$
This proves the claim.

Since also
$$
|x_0-g_i|\le r_i,
$$
it follows that
$$
[g_i-r_i,g_i+r_i]\subset [x_0-2r_i,x_0+2r_i].
$$
Therefore
$$
E\cap \left[g_i-\frac{1}{2q_i},g_i+\frac{1}{2q_i}\right]
\subset
[x_0-2r_i,x_0+2r_i].
$$

Because $a_n\to x_0$, for every sufficiently large $i$ there exists at least one index $n$ such that
$$
|a_n-x_0|<\frac{1}{2q_i}-r_i.
$$
Indeed, the right-hand side is positive for large $i$ because $r_i<1/(4q_i)$.

For such an index $n$, we have
$$
|a_n-g_i|\le |a_n-x_0|+|x_0-g_i|
<
\left(\frac{1}{2q_i}-r_i\right)+r_i
=
\frac{1}{2q_i}.
$$
Hence
$$
a_n\in E\cap \left[g_i-\frac{1}{2q_i},g_i+\frac{1}{2q_i}\right],
$$
and so by the inclusion proved above,
$$
|a_n-x_0|\le 2r_i.
$$

Now for each sufficiently large $i$, let $n(i)$ be the first index such that
$$
|a_{n(i)}-x_0|<\frac{1}{2q_i}-r_i.
$$
Then by minimality,
$$
|a_{n(i)-1}-x_0|\ge \frac{1}{2q_i}-r_i.
$$
On the other hand, from the previous paragraph,
$$
|a_{n(i)}-x_0|\le 2r_i.
$$
Therefore
$$
\frac{|a_{n(i)}-x_0|}{|a_{n(i)-1}-x_0|}
\le
\frac{2r_i}{\frac{1}{2q_i}-r_i}.
$$

For all sufficiently large $i$ we have $r_i<1/(4q_i)$, so
$$
\frac{1}{2q_i}-r_i\ge \frac{1}{4q_i}.
$$
Hence
$$
\frac{|a_{n(i)}-x_0|}{|a_{n(i)-1}-x_0|}
\le
8q_i r_i
=
8q_i^{1-\phi(q_i)}.
$$
Since $\phi(q_i)\to\infty$, the right-hand side tends to $0$. Thus
$$
\frac{|a_{n(i)}-x_0|}{|a_{n(i)-1}-x_0|}\to 0
$$
along the subsequence $n(i)$.

This contradicts the assumption that
$$
\frac{|a_{n+1}-x_0|}{|a_n-x_0|}\to 1.
$$

The contradiction proves the lemma.
\end{proof}

\begin{lemma}[There is a rapidly decreasing sequence  in the intersection]
\label{lemma:goodsequencesarethere}
Let
$$
q_1=2, \qquad q_{i+1}=q_i^{M_i},
$$
$$
G_i=\frac{1}{q_i}\big([0,q_i]\cap {\mathbb Z}\big),
$$
$$
r_i=q_i^{-\phi(q_i)},
$$
$$
E_i=\{x\in [0,1]: \operatorname{dist}(x,G_i)\le r_i\},
$$
and
$$
E=\bigcap_{i=1}^{\infty}E_i.
$$
Assume that
$$
\phi(q_i)<M_i-1
$$
for all sufficiently large $i$.

Then for all sufficiently large $i$,
$$
\frac{1}{q_i}\in E.
$$
In particular, $E$ contains a sequence
$$
a_i=\frac{1}{q_i}
$$
such that $a_i\to 0$. Moreover,
$$
\frac{a_{i+1}}{a_i}=\frac{1}{q_i^{M_i-1}}\to 0.
$$
\end{lemma}

\begin{remark} This lemma shows that in the regime $\phi(q_i)<M_i-1$, the set $E$ contains rapidly decaying sequences, for example
$$
a_i=\frac{1}{q_i}, \qquad \frac{a_{i+1}}{a_i}\to 0.
$$
In contrast, Lemma \ref{lemma:noslowsequences} shows that slowly decaying sequences cannot occur in $E$.
\end{remark} 

\begin{proof}

We claim that for all sufficiently large $i$,
$$
\frac{1}{q_i}\in E.
$$

Fix such an $i$. We check that
$$
\frac{1}{q_i}\in E_n
\qquad \text{for every } n.
$$

First suppose that $n\ge i$. Since
$$
q_{k+1}=q_k^{M_k},
$$
we have $q_i\mid q_n$. Therefore
$$
\frac{1}{q_i}=\frac{m}{q_n}
$$
for some integer $m$, which shows that
$$
\frac{1}{q_i}\in G_n.
$$
Hence
$$
\operatorname{dist}\left(\frac{1}{q_i},G_n\right)=0,
$$
and so
$$
\frac{1}{q_i}\in E_n.
$$

Now suppose that $n<i$. Since $0\in G_n$, it is enough to show that
$$
\frac{1}{q_i}\le r_n.
$$
Because $\phi$ is increasing and $q_n$ is increasing, the sequence
$$
r_n=q_n^{-\phi(q_n)}
$$
is decreasing. Therefore
$$
r_n\ge r_{i-1}
$$
for every $n<i$.

On the other hand,
$$
\frac{1}{q_i}=\frac{1}{q_{i-1}^{M_{i-1}}}.
$$
Since
$$
\phi(q_{i-1})<M_{i-1}-1<M_{i-1},
$$
we get
$$
\frac{1}{q_i}
=
q_{i-1}^{-M_{i-1}}
<
q_{i-1}^{-\phi(q_{i-1})}
=
r_{i-1}
\le r_n.
$$
Hence
$$
\operatorname{dist}\left(\frac{1}{q_i},G_n\right)\le \frac{1}{q_i}\le r_n,
$$
so
$$
\frac{1}{q_i}\in E_n.
$$

Thus
$$
\frac{1}{q_i}\in E_n
\qquad \text{for every } n,
$$
and therefore
$$
\frac{1}{q_i}\in \bigcap_{n=1}^{\infty}E_n = E.
$$

Hence for all sufficiently large $i$, the point
$$
a_i=\frac{1}{q_i}
$$
belongs to $E$.

Finally,
$$
\frac{a_{i+1}}{a_i}
=
\frac{q_i}{q_{i+1}}
=
\frac{q_i}{q_i^{M_i}}
=
\frac{1}{q_i^{M_i-1}}
\to 0
$$
as $i\to\infty$.

It follows that $E$ contains the sequence
$$
a_i=\frac{1}{q_i}
$$
with
$$
\frac{a_{i+1}}{a_i}\to 0.
$$
This completes the proof.
\end{proof}
\begin{lemma} (The assumptions of Bourgain's theorem are satisfied) \label{lemma:triplesum}
Let
$$
q_1=2, \qquad q_{n+1}=q_n^{M_n},
$$
$$
G_n=\frac{1}{q_n}\big([0,q_n]\cap {\mathbb Z}\big),
$$
$$
r_n=q_n^{-\phi(q_n)},
$$
$$
E_n=\{x\in [0,1]: \operatorname{dist}(x,G_n)\le r_n\},
$$
and
$$
E=\bigcap_{n=1}^{\infty} E_n.
$$
Assume that
$$
\phi(q_n)<M_n-1
$$
for all sufficiently large $n$.

Then there exists an infinite sequence
$$
a_1,a_2,a_3,\dots
$$
such that
$$
\{a_1,a_2,a_3,\dots\}\subset E
$$
and
$$
\{a_1,a_2,a_3,\dots\}+\{a_1,a_2,a_3,\dots\}+\{a_1,a_2,a_3,\dots\}\subset E.
$$
\end{lemma}

\begin{proof}
Choose an increasing sequence of positive integers
$$
2\le N_1<N_2<N_3<\dots
$$
and define
$$
a_k=\frac{1}{q_{N_k}}.
$$
We choose the integers $N_k$ inductively so that
$$
3a_1\le r_n
\qquad \text{for every } n<N_1,
$$
$$
6a_{k+1}\le r_n
\qquad \text{for every } n<N_{k+1},
$$
and
$$
a_{k+1}\le \frac12 a_k
\qquad \text{for every } k\ge 1.
$$
This is possible because $q_N\to\infty$ as $N\to\infty$.

Since
$$
a_{k+1}\le \frac12 a_k,
$$
we have
$$
\sum_{\ell=k+1}^\infty a_\ell \le 2a_{k+1}
$$
for every $k\ge 1$.
Hence
$$
3\sum_{\ell=k+1}^\infty a_\ell \le 6a_{k+1}.
$$

Let
$$
A=\{a_1,a_2,a_3,\dots\}.
$$
We will show that
$$
A\subset E
\qquad \text{and} \qquad
A+A+A\subset E.
$$

We first prove that
$$
A\subset E.
$$
Let
$$
a_m=\frac{1}{q_{N_m}}.
$$
We show that $a_m\in E_n$ for every $n$.

If $n\ge N_m$, then $q_{N_m}\mid q_n$, so
$$
a_m\in G_n,
$$
hence
$$
a_m\in E_n.
$$

Now suppose that $n<N_m$. Choose $t \in \mathbb{N}\cup\{0\}$ so that
$$
N_t\le n<N_{t+1} \text{ (case $n\in \mathbb{N}$)}, \, \text{ or } \, n<N_1 \text{ (case $t=0$)}.
$$
Then, $n< N_{t+1}\leq N_m$. Since $(N_k)_k$ is increasing, we have $m\ge t+1$, then
$$
a_m\le \sum_{\ell=t+1}^\infty a_\ell \le 2a_{t+1}.
$$
Since $n<N_{t+1}$, the choice of $N_{t+1}$ gives
$$
6a_{t+1}\le r_n \text{ (case $n\in \mathbb{N}$)}, \, \text{ or } 3a_1\leq r_n \text{ (case $t=0$).}
$$
so
$$
a_m\le 2a_{t+1}\le  r_n.
$$
Because $0\in G_n$, it follows that
$$
\operatorname{dist}(a_m,G_n)\le a_m\le r_n,
$$
and therefore
$$
a_m\in E_n.
$$

Thus $a_m\in E_n$ for every $n$, and so
$$
a_m\in E.
$$
Since $m$ was arbitrary, we have proved that
$$
A\subset E.
$$

Next we prove that
$$
A+A+A\subset E.
$$
Let
$$
x=a_i+a_j+a_k
$$
with $a_i,a_j,a_k\in A$.
We must show that $x\in E_n$ for every $n$.

If $n<N_1$, then
$$
x\le 3a_1\le r_n
$$
by the choice of $N_1$.
Since $0\in G_n$, we get
$$
\operatorname{dist}(x,G_n)\le x\le r_n,
$$
so
$$
x\in E_n.
$$

Now suppose that $n\ge N_1$. Choose $t$ so that
$$
N_t\le n<N_{t+1}.
$$
Write
$$
x=s_t+u_t,
$$
where $s_t$ is the sum of those terms among $a_i,a_j,a_k$ whose indices are at most $t$,
and $u_t$ is the sum of the remaining terms.

Every term contributing to $s_t$ has denominator dividing $q_{N_t}$, and since $N_t\le n$,
each such term belongs to $G_n$. Therefore
$$
s_t\in G_n.
$$

On the other hand, every term contributing to $u_t$ is at most
$$
\sum_{\ell=t+1}^\infty a_\ell,
$$
so
$$
0\le u_t\le 3\sum_{\ell=t+1}^\infty a_\ell \le 6a_{t+1}.
$$
Since $n<N_{t+1}$, the choice of $N_{t+1}$ gives
$$
6a_{t+1}\le r_n.
$$
Hence
$$
u_t\le r_n.
$$

Since $x=s_t+u_t$ and $s_t\in G_n$, it follows that
$$
\operatorname{dist}(x,G_n)\le u_t\le r_n.
$$
Therefore
$$
x\in E_n.
$$

Thus $x\in E_n$ for every $n$, and hence
$$
x\in E.
$$
We conclude that
$$
A+A+A\subset E.
$$

Therefore there exists an infinite sequence
$$
a_1,a_2,a_3,\dots
$$
such that
$$
\{a_1,a_2,a_3,\dots\}\subset E
$$
and
$$
\{a_1,a_2,a_3,\dots\}+\{a_1,a_2,a_3,\dots\}+\{a_1,a_2,a_3,\dots\}\subset E.
$$
\end{proof}
The previous lemma shows that the intersection set $E$ is generated by a highly branching lattice structure. Each lattice point at scale $q_i$ produces many descendants at scale $q_{i+1}$ that remain inside the neighborhood defining $E_i$. Because the lattices are nested, additive relations among lattice points persist across all scales. This branching additive structure allows one to construct an infinite set $A \subset E$ such that $A+A+A \subset E$. A theorem of Bourgain shows that any set containing $A+A+A$ for some infinite set $A$ satisfies the conclusion of the Erd\H{o}s similarity conjecture. Therefore, the sets constructed here satisfy the conjecture.

\begin{corollary}
Assume that
$$
\phi(q_n)<M_n-1
$$
for all sufficiently large $n$.
Then the set $E$ satisfies the conclusion of the Erd\H{o}s similarity conjecture.
\end{corollary}

\begin{proof}
By Lemma \ref{lemma:triplesum}, there exists an infinite sequence
$$
a_1,a_2,a_3,\dots
$$
such that
$$
\{a_1,a_2,a_3,\dots\}\subset E
$$
and
$$
\{a_1,a_2,a_3,\dots\}+\{a_1,a_2,a_3,\dots\}+\{a_1,a_2,a_3,\dots\}\subset E.
$$
The conclusion now follows from Bourgain's theorem \cite{Bourgain2003}.
\end{proof}

Before proving the next result, we introduce a new definition.

\begin{definition}
The logarithmic upper box dimension of a set $E$ is defined as
\[
\overline{\dim}^{\log}_{B}(E)
:=
\limsup_{\delta \to 0}
\frac{\log N_\delta(E)}{\log\!\big(\log(1/\delta)\big)},
\] where \(N_\delta(E)\) denotes the least number of sets of diameter at most \(\delta\)
needed to cover \(E\).
\end{definition}

\begin{lemma}
Let \(E \subset \mathbb{R}^d\) be a nonempty bounded set. Then
\[
\dim^{\log}_{P}(E)\leq \overline{\dim}^{\log}_{B}(E).
\]
\end{lemma}

\begin{proof}
Write
\[
\beta:=\overline{\dim}^{\log}_{B}(E).
\]

Fix any \(s>\beta\). Choose \(\gamma\) such that
\[
\beta<\gamma<s.
\]
By the definition of \(\overline{\dim}^{\log}_{B}(E)\), there exist \(C>0\) and
\(\delta_{0}\in(0,1/2)\) such that
\[
N_\delta(E)\leq C \bigl(\log(1/\delta)\bigr)^{\gamma}
\qquad \text{for all } 0<\delta<\delta_{0}.
\]
Let
\[
\phi(r):=(\log(1/r))^{-s}, \qquad 0<r<1.
\]
It is enough to prove that \(\mathcal P^{\phi}(E)=0\), since then, by the definition of
logarithmic packing dimension,
\[
\dim^{\log}_{P}(E)\le s,
\]
and letting \(s\downarrow\beta\) gives the claim.

We first show that \(\mathcal P^{\phi}_{0}(E)=0\). Let
\(\{B(x_i,r_i)\}_i\) be any \(\delta\)-packing of \(E\), with \(0<\delta<\delta_{0}/2\).
Set
\[
T:=\log(1/\delta).
\]
Since \(r_i<\delta\) for all \(i\), we have \(\log(1/r_i)>T\).

For \(t\ge T\), define
\[
F(t):=\#\{i: r_i\ge e^{-t}\}= \sum_i \chi_{\{t: \, t\geq \log(\frac{1}{r_i})\}}(t).
\]
If \(r_i\ge e^{-t}\), since the balls \(B(x_i,r_i)\) are disjoint, so their centers
\(x_i\) are pairwise at distance at least \(2e^{-t}\). Hence any set of diameter at most
\(2e^{-t}\) contains at most one such center. Therefore
\[
F(t)\le N_{2e^{-t}}(E).
\]
Since \(t\ge T>\log(2/\delta_0)\), we have \(2e^{-t}<\delta_0\), so
\[
F(t)\le C\bigl(\log(e^t/2)\bigr)^\gamma
= C(t-\log 2)^\gamma
\le C_1 t^\gamma
\]
for some constant \(C_1>0\) and all \(t\ge T\).

Now use the identity
\[
\phi(r_i)=(\log(1/r_i))^{-s}
=\int_{\log(1/r_i)}^\infty s\, t^{-s-1}\,dt.
\]
Summing over \(i\) and applying Tonelli's theorem,
\[
\sum_i \phi(r_i)
=
\sum_i \int_{\log(1/r_i)}^\infty s\, t^{-s-1}\,dt
=
\int_T^\infty s\, t^{-s-1} F(t)\,dt.
\]
Using the bound on \(F(t)\), we get
\[
\sum_i \phi(r_i)
\le
C_1 s \int_T^\infty t^{\gamma-s-1}\,dt
=
\frac{C_1 s}{s-\gamma}\, T^{\gamma-s}.
\]
Since \(\gamma<s\), the exponent \(\gamma-s\) is negative, and thus
\[
\frac{C_1 s}{s-\gamma}\, T^{\gamma-s}\to 0
\qquad \text{as } \delta\to 0
\]
(that is, as \(T\to\infty\)). Taking the supremum over all \(\delta\)-packings and then
letting \(\delta\to 0\), we conclude that
\[
\mathcal P^{\phi}_{0}(E)=0.
\]
Therefore also
\[
\mathcal P^{\phi}(E)=0.
\]

Since this holds for every \(s>\beta\), it follows from the definition of
\(\dim^{\log}_{P}(E)\) that
\[
\dim^{\log}_{P}(E)\le \beta
=
\overline{\dim}^{\log}_{B}(E).
\]
This completes the proof.
\end{proof}

\begin{proposition}[Small logarithmic packing dimension]\label{prop:logpackingless2}
Let
$$
q_1=2, \qquad q_{i+1}=q_i^{M_i},
$$
and define
$$
G_i=\frac{1}{q_i}\big([0,q_i]\cap {\mathbb Z}\big), \qquad
r_i=q_i^{-\phi(q_i)},
$$
$$
E_i=\{x\in[0,1]:\operatorname{dist}(x,G_i)\le r_i\},
\qquad
E=\bigcap_{i=1}^{\infty}E_i.
$$

Assume that $\phi$ is monotone increasing and tends to infinity.

Suppose there exist $\varepsilon>0$ and $C>0$ such that for all sufficiently large $n$,
$$
\prod_{j=1}^{n-1}\left(1+q_j^{M_j-\phi(q_j)}\right)
\le
C \left(\phi(q_{n-1})\log q_{n-1}\right)^{2-\varepsilon}.
$$

Then
$$
\dim_P^{\log}(E)\le 2-\varepsilon.
$$

In particular,
$$
\dim_P^{\log}(E)<2.
$$
\end{proposition}

\begin{proof}
For each $n$, let
$$
F_n=\bigcap_{j=1}^n E_j.
$$
Then
$$
E\subset F_n.
$$

We claim that there exists an absolute constant $C_1>0$ such that for every $n\ge 1$, the set $F_n$ can be covered by at most
$$
C_1^n \prod_{j=1}^{n-1}\left(1+q_j^{M_j-\phi(q_j)}\right)
$$
intervals of length at most $4r_n$.

For $n=1$, this is immediate, since $F_1=E_1$ is a finite union of intervals of length at most $2r_1$.

Assume inductively that $F_j$ can be covered by a collection of intervals of length at most $4r_j$. Let $I$ be one such interval. Since
$$
F_{j+1}=F_j\cap E_{j+1},
$$
the set $F_{j+1}\cap I$ is contained in the union of the intervals of radius $r_{j+1}$ around those points of $G_{j+1}$ whose distance to $I$ is at most $r_{j+1}$.

The $r_{j+1}$-neighborhood of $I$ has length at most
$$
4r_j+2r_{j+1}\le 6r_j,
$$
because $r_{j+1}\le r_j$. Since the points of $G_{j+1}$ are spaced by $1/q_{j+1}$, the number of points of $G_{j+1}$ contained in this neighborhood is bounded by a constant multiple of
$$
1+r_j q_{j+1}.
$$
Therefore $F_{j+1}\cap I$ can be covered by at most
$$
C_1\left(1+r_j q_{j+1}\right)
$$
intervals of length at most $2r_{j+1}$, hence also by the same number of intervals of length at most $4r_{j+1}$.

Since
$$
r_j q_{j+1}=q_j^{M_j-\phi(q_j)},
$$
it follows that each interval at stage $j$ produces at most
$$
C_1\left(1+q_j^{M_j-\phi(q_j)}\right)
$$
intervals at stage $j+1$. Iterating this estimate proves the claim.

Consequently,
$$
N(E,4r_n)\le N(F_n,4r_n)
\le
C_1^n \prod_{j=1}^{n-1}\left(1+q_j^{M_j-\phi(q_j)}\right).
$$

By hypothesis,
$$
N(E,4r_n)
\le
C_1^n C \left(\phi(q_{n-1})\log q_{n-1}\right)^{2-\varepsilon}
$$
for all sufficiently large $n$.

Since $q_n=q_{n-1}^{M_{n-1}}$ and $M_{n-1}\to\infty$, the sequence $q_n$ grows at least doubly exponentially for all sufficiently large $n$. Hence
$$
\log q_{n-1}
$$
grows at least exponentially in $n$. Because $\phi(q_{n-1})\to\infty$, the quantity
$$
\phi(q_{n-1})\log q_{n-1}
$$
eventually dominates every exponential factor $C_1^n$. Therefore, after increasing the constant if necessary, we obtain
$$
N(E,4r_n)
\le
C_2 \left(\phi(q_{n-1})\log q_{n-1}\right)^{2-\varepsilon}
$$
for all sufficiently large $n$.

Now let $\delta>0$ be sufficiently small, and choose $n$ so that
$$
4r_n\le \delta<4r_{n-1}.
$$
By monotonicity of covering numbers,
$$
N(E,\delta)\le N(E,4r_n).
$$
Therefore
$$
N(E,\delta)\le
C_2 \left(\phi(q_{n-1})\log q_{n-1}\right)^{2-\varepsilon}.
$$

Since
$$
\delta<4r_{n-1},
$$
we have
$$
\log\frac1\delta \ge \log\frac1{4r_{n-1}}.
$$
For all sufficiently large $n$,
$$
\log\frac1{4r_{n-1}} \ge \frac12 \log\frac1{r_{n-1}}
=
\frac12 \phi(q_{n-1})\log q_{n-1}.
$$
Hence
$$
\phi(q_{n-1})\log q_{n-1}
\le
2\log\frac1\delta,
$$
and so
$$
N(E,\delta)\le
C_3\left(\log\frac1\delta\right)^{2-\varepsilon}
$$
for all sufficiently small $\delta$.

It follows that the logarithmic upper box dimension of $E$ is at most $2-\varepsilon$. Since logarithmic packing dimension is bounded above by logarithmic upper box dimension, we conclude that
$$
\dim_P^{\log}(E)\le 2-\varepsilon.
$$

In particular,
$$
\dim_P^{\log}(E)<2.
$$
\end{proof}

We now record a logarithmic Hausdorff dimension analogue of Proposition 2.8 under a stronger growth assumption.

\begin{proposition}[Small logarithmic Hausdorff dimension]
Let
$$
q_1=2, \qquad q_{i+1}=q_i^{M_i},
$$
and define
$$
G_i=\frac{1}{q_i}\big([0,q_i]\cap {\mathbb Z}\big), \qquad
r_i=q_i^{-\phi(q_i)},
$$
$$
E_i=\{x\in[0,1]:\operatorname{dist}(x,G_i)\le r_i\},
\qquad
E=\bigcap_{i=1}^{\infty}E_i.
$$

Assume that $\phi$ is monotone increasing and tends to infinity.

Suppose there exist $\varepsilon>0$ and $C>0$ such that for all sufficiently large $n$,
$$
\prod_{j=1}^{n-1}\left(1+q_j^{M_j-\phi(q_j)}\right)
\le
C \left(\phi(q_n)\log q_n\right)^{1-\varepsilon}.
$$

Then
$$
\dim_H^{\log}(E)\le 1-\varepsilon.
$$

In particular,
$$
\dim_H^{\log}(E)<1.
$$
\end{proposition}

\begin{proof}
For each $n$, let
$$
F_n=\bigcap_{j=1}^n E_j.
$$
Then
$$
E\subset F_n.
$$

Exactly as in the proof of Proposition 2.8, there exists an absolute constant $C_1>0$ such that $F_n$ can be covered by at most
$$
C_1^n \prod_{j=1}^{n-1}\left(1+q_j^{M_j-\phi(q_j)}\right)
$$
intervals of length at most $4r_n$.

Therefore, by the hypothesis, for all sufficiently large $n$ the set $E$ can be covered by at most
$$
C_1^n C \left(\phi(q_n)\log q_n\right)^{1-\varepsilon}
$$
intervals of length at most $4r_n$.

Fix $s>1-\varepsilon$. We will show that the logarithmic Hausdorff $s$-measure of $E$ is zero.

Let
$$
h_s(r)=\left(\log\frac{1}{r}\right)^{-s}
$$
for sufficiently small $r>0$.
Since
$$
r_n=q_n^{-\phi(q_n)},
$$
we have
$$
\log\frac{1}{r_n}=\phi(q_n)\log q_n.
$$

Because $q_n\to\infty$ and $\phi(q_n)\to\infty$, it follows that
$$
\phi(q_n)\log q_n\to\infty.
$$
Also, since
$$
q_{n+1}=q_n^{M_n}
$$
and
$$
M_n\to\infty,
$$
the quantity
$$
\phi(q_n)\log q_n
$$
eventually dominates every exponential factor. In particular, if we set
$$
\delta=\frac{s-(1-\varepsilon)}{2}>0,
$$
then for all sufficiently large $n$,
$$
C_1^n\le \left(\phi(q_n)\log q_n\right)^\delta.
$$

Hence for all sufficiently large $n$,
$$
C_1^n C \left(\phi(q_n)\log q_n\right)^{1-\varepsilon}
\le
C \left(\phi(q_n)\log q_n\right)^{1-\varepsilon+\delta}.
$$

Now consider the above cover of $E$ by intervals $I$ with
$$
|I|\le 4r_n.
$$
Since $h_s$ is increasing for sufficiently small $r$, we get
$$
h_s(|I|)\le h_s(4r_n)
$$
for every interval in the cover. Therefore the $h_s$-cost of this cover is bounded by
$$
C_1^n C \left(\phi(q_n)\log q_n\right)^{1-\varepsilon} h_s(4r_n).
$$

We next estimate $h_s(4r_n)$. Since
$$
\log\frac{1}{4r_n}=\log\frac{1}{r_n}-\log 4
=\phi(q_n)\log q_n-\log 4,
$$
for all sufficiently large $n$ we have
$$
\log\frac{1}{4r_n}\ge \frac{1}{2}\phi(q_n)\log q_n.
$$
Hence
$$
h_s(4r_n)
=
\left(\log\frac{1}{4r_n}\right)^{-s}
\le
2^s \left(\phi(q_n)\log q_n\right)^{-s}.
$$

It follows that
$$
\mathcal H^{h_s}_{4r_n}(E)
\le
C_1^n C \left(\phi(q_n)\log q_n\right)^{1-\varepsilon} 2^s
\left(\phi(q_n)\log q_n\right)^{-s}.
$$
Using
$$
C_1^n\le \left(\phi(q_n)\log q_n\right)^\delta
$$
for large $n$, we obtain
$$
\mathcal H^{h_s}_{4r_n}(E)
\le
C' \left(\phi(q_n)\log q_n\right)^{1-\varepsilon+\delta-s}.
$$
By the choice of $\delta$,
$$
1-\varepsilon+\delta-s
=
-\delta<0.
$$
Therefore
$$
\mathcal H^{h_s}_{4r_n}(E)\to 0
$$
as $n\to\infty$.

Hence
$$
\mathcal H^{h_s}(E)=0
$$
for every $s>1-\varepsilon$.
This implies that
$$
\dim_H^{\log}(E)\le 1-\varepsilon.
$$

In particular,
$$
\dim_H^{\log}(E)<1.
$$
\end{proof}

The construction exhibits a threshold governed by the comparison between $\phi(q_i)$ and $M_i$. Indeed, since
$$
q_{i+1}=q_i^{M_i},
$$
the spacing of the next lattice is
$$
\frac{1}{q_{i+1}}=q_i^{-M_i},
$$
while the radius at stage $i$ is
$$
r_i=q_i^{-\phi(q_i)}.
$$
If
$$
\phi(q_i)>M_i
$$
eventually, then the neighborhoods are too small to capture genuinely new descendants at the next scale, and the intersection collapses to a finite set. If
$$
\phi(q_i)<M_i-1
$$
eventually, then each surviving interval contains many descendants at the next scale, which leads to an uncountable intersection and, in fact, to the existence of an infinite set $A$ with
$$
A+A+A\subset E.
$$
Thus the behavior is governed by the relation between the decay exponent $\phi(q_i)$ and the lattice growth exponent $M_i$.

\begin{corollary}\label{cor:explicitlogpacking}
There exist choices of the sequence $(M_n)$ and the monotone increasing function $\phi$ with $\phi(t)\to\infty$ such that
$$
\phi(q_n)<M_n-1
$$
for every $n$ and
$$
\dim_P^{\log}(E)\le 1.
$$
In particular,
$$
\dim_P^{\log}(E)<2.
$$
\end{corollary}

\begin{proof}
Set
$$
q_1=2.
$$
Choose
$$
M_1=\max\left\{4,\ \left\lceil \frac{2(1+q_1^2)}{\log q_1}\right\rceil+2\right\}.
$$
For each $n\ge 2$, once $q_n$ has been defined, choose
$$
M_n=\max\left\{M_{n-1}+1,\ \left\lceil \frac{2}{\log q_n}\prod_{j=1}^{n}(1+q_j^2)\right\rceil+2\right\}.
$$
Then define
$$
q_{n+1}=q_n^{M_n}.
$$

Define $\phi:[2,\infty)\to(0,\infty)$ by
$$
\phi(t)=M_n-2
\qquad \text{for } q_n\le t<q_{n+1}.
$$

Since
$$
M_n-\phi(q_n)=2,
$$
we have
$$
q_n^{M_n-\phi(q_n)}=q_n^2,
$$
so
$$
\prod_{j=1}^{n-1}\left(1+q_j^{M_j-\phi(q_j)}\right)
=
\prod_{j=1}^{n-1}(1+q_j^2).
$$

By the definition of $M_{n-1}$,
$$
\prod_{j=1}^{n-1}(1+q_j^2)\le \frac{1}{2}\phi(q_{n-1})\log q_{n-1}.
$$

Thus the hypothesis of Proposition \ref{prop:logpackingless2} holds, and
$$
\dim_P^{\log}(E)\le 1.
$$
\end{proof}


\begin{corollary}
Under the hypotheses of Proposition \ref{prop:logpackingless2}, the set $E$ has logarithmic packing dimension strictly less than $2$. In particular, the construction falls outside the positive-logarithmic-dimension regime considered in \cite{ShmerkinYavicoli2025}.
\end{corollary}

\begin{proof}
This follows immediately from Proposition \ref{prop:logpackingless2}.
\end{proof}

\section{Bourgain's result on sumsets is not enough to prove the Erdos conjecture for Cantor sets}

\begin{proposition} Let \((\rho_n)_{n\ge 1}\) be a sequence of positive real numbers such that \[ \rho_n \downarrow 0, \qquad \rho_{n+1}<\frac{\rho_n}{10}\quad\text{for all }n, \] and \[ 2^n(\log(1/\rho_n))^{-s}\longrightarrow 0 \qquad\text{for every }s>0. \] Then there exists a Cantor set \(K\subset (1,2)\) such that: \begin{enumerate} \item \(K\) is linearly independent over \(\mathbb{Q}\); \item for each \(n\), \(K\) is covered by \(2^n\) closed intervals of diameter at most \(\rho_n\); \item \(\dim_H^{\log}(K)=0\). \end{enumerate} In particular, there is a whole family of such sets, obtained by varying the sequence \((\rho_n)\). \end{proposition} 

\begin{proof} For \(s>0\), let \[ h_s(r):=(\log(1/r))^{-s}\qquad (0<r<e^{-1}), \] and let \(\mathcal H^{h_s}\) be the Hausdorff measure associated to this gauge. By definition, \[ \dim_H^{\log}(E):=\inf\{s>0:\mathcal H^{h_s}(E)=0\}. \] We will construct \(K\) by a binary Cantor scheme. 

\medskip

For each \(m\ge 1\), let \[ \mathcal L_m = \left\{ L(x_1,\dots,x_m)=q_1x_1+\cdots+q_mx_m : q_1,\dots,q_m\in\mathbb Q\setminus\{0\} \right\}. \] Since \(\mathbb Q\) is countable, the set \[ \mathcal L:=\bigcup_{m\ge 1}\mathcal L_m \] is countable. Fix an enumeration \[ \mathcal L=\{L_1,L_2,L_3,\dots\}. \] For each \(k\), let \(m_k\) denote the number of variables of \(L_k\), so that \[ L_k(x_1,\dots,x_{m_k})=q_{k,1}x_1+\cdots+q_{k,m_k}x_{m_k}, \qquad q_{k,i}\in\mathbb Q\setminus\{0\}. \]

\medskip 

\textbf{Claim:} Let \(J_1,\dots,J_N\subset\mathbb R\) be nonempty open intervals. Then one can choose points \[ x_i\in J_i\qquad (1\le i\le N) \] such that the set \(\{x_1,\dots,x_N\}\) is linearly independent over \(\mathbb Q\). 

The claim is clear, since one can proceed inductively: Choose \(x_1\in J_1\setminus\{0\}\). Suppose \(x_1,\dots,x_{r-1}\) have already been chosen and are linearly independent over \(\mathbb Q\). Then \[ \operatorname{span}_{\mathbb Q}(x_1,\dots,x_{r-1}) \] is countable, because it consists of all finite rational linear combinations of the \(x_i\). Since \(J_r\) is an uncountable interval, we may choose \[ x_r\in J_r\setminus \operatorname{span}_{\mathbb Q}(x_1,\dots,x_{r-1}). \] Then \(x_1,\dots,x_r\) remain linearly independent over \(\mathbb Q\).

\medskip

For a finite binary word \(u\in\{0,1\}^n\), write \(|u|=n\). We will construct, for each finite binary word \(u\), a nonempty closed interval \(I_u\subset (1,2)\) such that the following properties hold: 

\begin{enumerate} 
\item[(a)] \(I_{\varnothing}\subset (1,2)\), where \(\varnothing\) denotes the empty word; 
\item[(b)] if \(u\in\{0,1\}^n\), then \(I_{u0}\) and \(I_{u1}\) are disjoint closed subintervals of \(\operatorname{int}(I_u)\); 
\item[(c)] if \(|u|=n\ge 1\), then \(\operatorname{diam}(I_u)\le \rho_n\); 
\item[(d)] for every \(n\ge 1\), every \(k\le n\), and every choice of pairwise distinct words \[ u_1,\dots,u_{m_k}\in\{0,1\}^n, \] we have \[ 0\notin L_k(I_{u_1}\times\cdots\times I_{u_{m_k}}). \] \end{enumerate} Once these intervals are constructed, define \[ K:=\bigcap_{n=1}^{\infty}\ \bigcup_{u\in\{0,1\}^n} I_u. \] We now perform the construction by induction on the level \(n\). 

\medskip \noindent\emph{Initial step.}

Choose any nondegenerate closed interval \[ I_{\varnothing}\subset (1,2). \] \medskip \noindent\emph{Inductive step.} Assume that all intervals \(I_u\) with \(|u|\le n-1\) have been constructed and satisfy (a)--(d) up to level \(n-1\). We explain how to define the intervals \(I_v\) with \(|v|=n\). For each word \(u\in\{0,1\}^{n-1}\), the interior \(\operatorname{int}(I_u)\) is a nonempty open interval. There are \(2^n\) words of length \(n\), namely the two children \(u0,u1\) of each \(u\in\{0,1\}^{n-1}\). Applying the lemma to these \(2^n\) open intervals, we may choose points \[ x_v\in \operatorname{int}(I_{\operatorname{par}(v)}) \qquad (v\in\{0,1\}^n), \] where \(\operatorname{par}(v)\) denotes the parent of \(v\), in such a way that the finite set \[ X_n:=\{x_v:v\in\{0,1\}^n\} \] is linearly independent over \(\mathbb Q\). Now fix \(k\le n\), and let \[ u_1,\dots,u_{m_k}\in\{0,1\}^n \] be pairwise distinct. Since \(X_n\) is linearly independent over \(\mathbb Q\), we have \[ L_k(x_{u_1},\dots,x_{u_{m_k}})\neq 0. \] 

There are only finitely many pairs \[ \bigl(k,(u_1,\dots,u_{m_k})\bigr) \] with \(k\le n\) and \(u_1,\dots,u_{m_k}\in\{0,1\}^n\) pairwise distinct. Therefore, by continuity of each linear form \(L_k\), we may choose \(\varepsilon_n>0\) so small that: 

\begin{enumerate} 
\item[(i)] if \(v,w\in\{0,1\}^n\) have the same parent and \(v\neq w\), then \[ |x_v-x_w|>4\varepsilon_n; \] 
\item[(ii)] for every \(v\in\{0,1\}^n\), \[ [x_v-\varepsilon_n,x_v+\varepsilon_n]\subset \operatorname{int}(I_{\operatorname{par}(v)}); \] 
\item[(iii)] \(2\varepsilon_n\le \rho_n\); 
\item[(iv)] for every \(k\le n\) and every pairwise distinct \[ u_1,\dots,u_{m_k}\in\{0,1\}^n, \] we have \[ 0\notin L_k\bigl([x_{u_1}-\varepsilon_n,x_{u_1}+\varepsilon_n]\times\cdots\times [x_{u_{m_k}}-\varepsilon_n,x_{u_{m_k}}+\varepsilon_n]\bigr). \] \end{enumerate} 

Now define \[ I_v:=[x_v-\varepsilon_n,x_v+\varepsilon_n] \qquad (v\in\{0,1\}^n). \] Then \(I_v\subset \operatorname{int}(I_{\operatorname{par}(v)})\) by (ii), sibling intervals are disjoint by (i), \(\operatorname{diam}(I_v)=2\varepsilon_n\le \rho_n\) by (iii), and property (d) holds at level \(n\) by (iv). This completes the induction.

\medskip \noindent
Let's see that \(K\) is homeomorphic to the standard Cantor space \(\{0,1\}^{\mathbb N}\).

Let \(\alpha=(\alpha_1,\alpha_2,\dots)\in\{0,1\}^{\mathbb N}\), and let \(\alpha|n\) be its initial word of length \(n\). Then \[ I_{\alpha|1}\supset I_{\alpha|2}\supset I_{\alpha|3}\supset\cdots \] is a nested sequence of nonempty compact intervals. Since \[ \operatorname{diam}(I_{\alpha|n})\le \rho_n\longrightarrow 0, \] the intersection \[ \bigcap_{n=1}^{\infty} I_{\alpha|n} \] consists of exactly one point. Denote this point by \(\pi(\alpha)\). Thus we obtain a map \[ \pi:\{0,1\}^{\mathbb N}\to K. \] It is injective, because if \(\alpha\neq \beta\), then for some \(n\) the words \(\alpha|n\) and \(\beta|n\) are distinct, and the corresponding level-\(n\) intervals are disjoint. It is surjective by the definition of \(K\). It is continuous, because if two sequences agree on the first \(n\) digits, then their images lie in the same interval \(I_u\) with \(|u|=n\), whose diameter is at most \(\rho_n\to 0\). Since \(\{0,1\}^{\mathbb N}\) is compact and \(\mathbb R\) is Hausdorff, \(\pi\) is a homeomorphism. Therefore \(K\) is compact, perfect, and totally disconnected; that is, \(K\) is a Cantor set. 

\medskip 

Take distinct points \(y_1,\dots,y_m\in K\) and nonzero rational numbers \(q_1,\dots,q_m\in\mathbb Q\). We will prove that \[ q_1y_1+\cdots+q_my_m\neq 0. \] Let \[ L(y_1,\dots,y_m):=q_1y_1+\cdots+q_my_m. \] Then \(L=L_k\) for some \(k\). Because the points \(y_1,\dots,y_m\) are distinct and because the diameters of the level-\(n\) intervals tend to \(0\), there exists \(n\ge k\) such that the points \(y_1,\dots,y_m\) lie in pairwise distinct level-\(n\) intervals \[ I_{u_1},\dots,I_{u_m}. \] Property (d) at level \(n\) gives \[ 0\notin L_k(I_{u_1}\times\cdots\times I_{u_m}). \] Since \(y_i\in I_{u_i}\) for each \(i\), it follows that \[ L(y_1,\dots,y_m)\neq 0. \] Thus \[ q_1y_1+\cdots+q_my_m\neq 0. \] So \(K\) is linearly independent over \(\mathbb Q\). 

\medskip 
Let's see that the set has logarithmic dimension $0$.

Fix \(s>0\). At level \(n\), the set \(K\) is covered by the \(2^n\) intervals \(\{I_u:|u|=n\}\), and each such interval has diameter at most \(\rho_n\). Therefore \[ \mathcal H^{h_s}_{\rho_n}(K) \le \sum_{|u|=n} h_s(\operatorname{diam}(I_u)) \le 2^n h_s(\rho_n) = 2^n(\log(1/\rho_n))^{-s}. \] By assumption, \[ 2^n(\log(1/\rho_n))^{-s}\longrightarrow 0. \] Hence \(\mathcal H^{h_s}(K)=0\) for every \(s>0\), and therefore \[ \dim_H^{\log}(K)=0. \] This completes the proof. \end{proof} \begin{corollary} There exists a Cantor set \(K\subset\mathbb R\) which is linearly independent over \(\mathbb Q\) and satisfies \[ \dim_H^{\log}(K)=0. \] \end{corollary} \begin{proof} Apply the proposition with \[ \rho_n:=\exp(-e^{n^2}). \] Then \(\rho_n\downarrow 0\), and for all sufficiently large \(n\), \[ \rho_{n+1}<\frac{\rho_n}{10}. \] Also, for every \(s>0\), \[ 2^n(\log(1/\rho_n))^{-s} = 2^n e^{-sn^2}\longrightarrow 0. \] Hence the proposition applies. 
\end{proof}
\medskip

\begin{remark} The proposition gives a whole family of examples. Indeed, any sequence \((\rho_n)\) tending to \(0\) sufficiently fast and satisfying \[ 2^n(\log(1/\rho_n))^{-s}\to 0 \qquad\text{for every }s>0 \] produces such a Cantor set. For example, one may take \[ \rho_n=\exp(-e^{n^\alpha})\qquad (\alpha>1), \] or \[ \rho_n=\exp(-e^{2^n}), \] or many other superlogarithmically decaying scales. \end{remark}

\begin{proposition} Let \(K\subset\mathbb R\) be linearly independent over \(\mathbb Q\). Then there do not exist infinite sets \(S_1,S_2,S_3\subset\mathbb R\) such that \[ S_1+S_2+S_3\subset K. \] In particular, none of the Cantor sets constructed above can contain a set of the form \(S_1+S_2+S_3\) with each \(S_i\) infinite. 
\end{proposition}

\begin{proof} Suppose for contradiction that there exist infinite sets \(S_1,S_2,S_3\subset\mathbb R\) such that \[ S_1+S_2+S_3\subset K. \] Since \(S_1, \, S_2, S_3\) are infinite, we can choose \(a\in S_1\), choose two distinct elements \[ b_1,b_2\in S_2,\qquad b_1\neq b_2, \] and two distinct elements \[ c_1,c_2\in S_3,\qquad c_1\neq c_2. \] Define \[ x_{11}=a+b_1+c_1,\qquad x_{12}=a+b_1+c_2,\qquad x_{21}=a+b_2+c_1,\qquad x_{22}=a+b_2+c_2. \] By assumption, each \(x_{ij}\in K\). These four numbers satisfy the identity \[ x_{11}+x_{22}=x_{12}+x_{21}, \] because both sides are equal to \[ 2a+b_1+b_2+c_1+c_2. \] Equivalently, \[ x_{11}-x_{12}-x_{21}+x_{22}=0. \] We claim that the four numbers \(x_{11},x_{12},x_{21},x_{22}\) are pairwise distinct. Indeed: \[\text{If } x_{11}=x_{12}\ \Longrightarrow\ c_1=c_2, \] which is impossible. \[\text{If } x_{11}=x_{21}\ \Longrightarrow\ b_1=b_2, \] which is impossible. \[\text{If } x_{11}=x_{22}\ \Longrightarrow\ b_1+c_1=b_2+c_2. \] But then the identity \[ x_{11}+x_{22}=x_{12}+x_{21} \] would imply \[ 2x_{11}=x_{12}+x_{21}. \] Since \(x_{12}\neq x_{21}\) (otherwise \(b_1+c_2=b_2+c_1\), hence \(b_1-b_2=c_1-c_2\), and together with \(b_1+c_1=b_2+c_2\) this gives \(b_1=b_2\), contradiction), this is a nontrivial rational linear relation among distinct points of \(K\), impossible. Hence \(x_{11}\neq x_{22}\). The remaining equalities are similar. Thus \(x_{11},x_{12},x_{21},x_{22}\) are four distinct points of \(K\), and they satisfy the nontrivial rational linear relation \[ x_{11}-x_{12}-x_{21}+x_{22}=0. \] This contradicts the assumption that \(K\) is linearly independent over \(\mathbb Q\). Therefore no such infinite sets \(S_1,S_2,S_3\) can exist. \end{proof}

\begin{corollary} There exists a whole family of Cantor sets \(K\subset\mathbb R\) such that: 
\begin{enumerate} \item \(K\) is linearly independent over \(\mathbb Q\); \item \(\dim_H^{\log}(K)=0\); \item there do not exist infinite sets \(S_1,S_2,S_3\subset\mathbb R\) with \[ S_1+S_2+S_3\subset K. \] \end{enumerate} 
\end{corollary}

\section{A large family of Cantor sets of logarithmic Hausdorff dimension zero containing a triple sumset} 
\label{section:positiveexamples}

The purpose of this section is to show that variants of the above construction can have extremely small logarithmic Hausdorff dimension and contain a triple infinite sumset (then verifying the Erdős conjecture), further emphasizing that the mechanism developed in this paper lies outside the positive logarithmic dimension regime.

We use the logarithmic Hausdorff dimension associated to the gauges \[ h_s(r):=(\log(1/r))^{-s}\qquad (0<r<e^{-1},\ s>0), \] and \[ \dim_H^{\log}(E):=\inf\{s>0:\mathcal H^{h_s}(E)=0\}. \] This is the logarithmic Hausdorff dimension considered by Shmerkin and Yavicoli. \cite{ShmerkinYavicoli2025} \begin{proposition}\label{prop:main-triple-sum} Let \((a_n)_{n\ge 1}\) be a sequence of positive real numbers such that \begin{equation}\label{eq:separation} a_n>\sum_{m>n} a_m \qquad\text{for every }n\ge 1, \end{equation} and \begin{equation}\label{eq:log-zero-condition} 2^n\Bigl(\log\frac{1}{\tau_n}\Bigr)^{-s}\longrightarrow 0 \qquad\text{for every }s>0, \end{equation} where \[ \tau_n:=\sum_{m>n} a_m. \] Define \[ K:=\left\{\sum_{n=1}^{\infty}\varepsilon_n a_n:\ \varepsilon_n\in\{0,1\}\ \text{for all }n\right\}. \] Then the following hold: \begin{enumerate} \item \(K\subset \mathbb R\) is a Cantor set. \item \(\dim_H^{\log}(K)=0\). \item There exist three countably infinite sets \(S_1,S_2,S_3\subset K\) such that \[ S_1+S_2+S_3\subset K. \] Equivalently, there exist three infinite sequences \((x_n^{(1)})_{n\ge 1}\), \((x_n^{(2)})_{n\ge 1}\), \((x_n^{(3)})_{n\ge 1}\) such that \[ \{x_n^{(1)}:n\ge 1\}+\{x_n^{(2)}:n\ge 1\}+\{x_n^{(3)}:n\ge 1\}\subset K. \] \end{enumerate} Hence there is a whole family of Cantor sets in \(\mathbb R\) of logarithmic Hausdorff dimension zero that contain a triple sumset of three infinite sequences, then satisfying the Erdős conjecture. \end{proposition} 

\begin{proof}
\medskip 
Let's see that the series defining \(K\) converges, so \(K\) is well defined.

Since the numbers \(a_n\) are positive and satisfy \eqref{eq:separation}, we have \[ a_n>\sum_{m>n}a_m=\tau_n \] for every \(n\). In particular the series $\sum a_n$ converges absolutely, so the series \[ \sum_{n=1}^{\infty}\varepsilon_n a_n \] converges absolutely. 

\medskip 
Let's see that $K$ is a Cantor set.

Consider the map \[ \pi:\{0,1\}^{\mathbb N}\to \mathbb R, \qquad \pi((\varepsilon_n)_{n\ge 1})=\sum_{n=1}^{\infty}\varepsilon_n a_n. \] Since \(\sum_n a_n<\infty\), the series converges uniformly with respect to \((\varepsilon_n)\in\{0,1\}^{\mathbb N}\). Therefore \(\pi\) is continuous. Its image is exactly \(K\). Since \(\{0,1\}^{\mathbb N}\) is compact, it follows already that \(K\) is compact. We now prove that \(\pi\) is injective. Let \((\varepsilon_n)\neq (\eta_n)\) in \(\{0,1\}^{\mathbb N}\), and let \(k\) be the first index at which they differ. Then \[ \varepsilon_j=\eta_j\quad (j<k), \qquad \varepsilon_k-\eta_k\in\{\pm 1\}. \] Hence \[ \pi((\varepsilon_n))-\pi((\eta_n)) = (\varepsilon_k-\eta_k)a_k+\sum_{m>k}(\varepsilon_m-\eta_m)a_m. \] Taking absolute values and using \(|\varepsilon_m-\eta_m|\le 1\), we obtain \[ \bigl|\pi((\varepsilon_n))-\pi((\eta_n))\bigr| \ge a_k-\sum_{m>k}a_m = a_k-\tau_k >0 \] by \eqref{eq:separation}. Therefore \(\pi((\varepsilon_n))\neq \pi((\eta_n))\), so \(\pi\) is injective. Since \(\{0,1\}^{\mathbb N}\) is compact and \(\mathbb R\) is Hausdorff, the continuous bijection \(\pi:\{0,1\}^{\mathbb N}\to K\) is a homeomorphism. Thus \(K\) is homeomorphic to the standard Cantor space \(\{0,1\}^{\mathbb N}\). In particular, \(K\) is a Cantor set, i.e. it is compact, perfect, and totally disconnected. 

\medskip 
Now, let's look at the cylinder sets and their diameters. For a finite word \(u=(u_1,\dots,u_n)\in\{0,1\}^n\), define the level-\(n\) cylinder \[ K_u := \left\{ \sum_{j=1}^{n}u_j a_j+\sum_{m>n}\varepsilon_m a_m:\ \varepsilon_m\in\{0,1\} \right\}. \] Then \[ K=\bigcup_{u\in\{0,1\}^n} K_u \qquad\text{for every }n. \] There are exactly \(2^n\) such cylinders. Fix \(u\in\{0,1\}^n\). If \(x,y\in K_u\), then the first \(n\) digits of \(x\) and \(y\) agree, so \[ x-y=\sum_{m>n} (\varepsilon_m-\eta_m)a_m \] for suitable \(\varepsilon_m,\eta_m\in\{0,1\}\). Therefore \[ |x-y| \le \sum_{m>n} a_m =\tau_n. \] Hence \[ \operatorname{diam}(K_u)\le \tau_n. \] So for every \(n\), the set \(K\) is covered by \(2^n\) sets of diameter at most \(\tau_n\). 

\medskip 
Let's prove that the logarithmic Hausdorff dimension of \(K\) is zero.

Fix \(s>0\). Let \(h_s(r)=(\log(1/r))^{-s}\). Since \(K\) is covered by the \(2^n\) cylinders \(K_u\) of level \(n\), each having diameter at most \(\tau_n\), we obtain \[ \mathcal H^{h_s}_{\tau_n}(K) \le \sum_{u\in\{0,1\}^n} h_s(\operatorname{diam}(K_u)) \le 2^n h_s(\tau_n) = 2^n\Bigl(\log\frac{1}{\tau_n}\Bigr)^{-s}. \] By assumption \eqref{eq:log-zero-condition}, \[ 2^n\Bigl(\log\frac{1}{\tau_n}\Bigr)^{-s}\longrightarrow 0. \] Therefore \[ \mathcal H^{h_s}(K)=0 \qquad\text{for every }s>0. \] By the definition of logarithmic Hausdorff dimension, \[ \dim_H^{\log}(K)=0. \] 

\medskip 

It remains to check that there are three infinite sets whose triple sumset lies in \(K\).

Choose a partition of \(\mathbb N\) into three pairwise disjoint infinite sets: \[ \mathbb N=A_1\sqcup A_2\sqcup A_3. \] For \(i=1,2,3\), define \[ S_i:= \left\{ \sum_{n\in F} a_n:\ F\subset A_i,\ F\ \text{finite} \right\}. \] We claim that each \(S_i\) is countably infinite. First, \(S_i\) is countable because the family of finite subsets of the countable set \(A_i\) is countable. Second, \(S_i\) is infinite because \(A_i\) is infinite and for each \(n\in A_i\), \[ a_n=\sum_{m\in \{n\}} a_m\in S_i. \] Since the \(a_n\) are all distinct positive numbers, this gives infinitely many distinct elements of \(S_i\). Now let \(x_i\in S_i\). Then there exist finite sets \(F_i\subset A_i\) such that \[ x_i=\sum_{n\in F_i} a_n \qquad (i=1,2,3). \] Because the sets \(A_1,A_2,A_3\) are pairwise disjoint, the finite sets \(F_1,F_2,F_3\) are pairwise disjoint as well. Hence \[ x_1+x_2+x_3 = \sum_{n\in F_1\cup F_2\cup F_3} a_n. \] This is a sum with coefficients \(0\) or \(1\), so it belongs to \(K\). Therefore \[ S_1+S_2+S_3\subset K. \] Finally, since each \(S_i\) is countably infinite, it can be enumerated as a sequence. Thus there exist infinite sequences \((x_n^{(i)})_{n\ge 1}\) such that \[ \{x_n^{(1)}:n\ge 1\}+\{x_n^{(2)}:n\ge 1\}+\{x_n^{(3)}:n\ge 1\}\subset K. \] 
\end{proof}

\begin{remark}\label{rem:many-examples} Proposition \ref{prop:main-triple-sum} gives a large family of examples satisfying the Erdős conjecture. For instance, if \[ a_n=\exp(-e^{n^2}), \] then \(\sum_{m>n}a_m<a_n\) for every sufficiently large \(n\), and after modifying finitely many initial terms one gets a sequence satisfying \eqref{eq:separation}. Moreover, for every \(s>0\), \[ 2^n\Bigl(\log\frac{1}{\tau_n}\Bigr)^{-s}\to 0, \] because \(\tau_n\) is comparable to \(a_{n+1}\) for such rapidly decreasing sequences. The same is true for many other super-fast decays, for example \[ a_n=\exp(-e^{n^\alpha})\qquad (\alpha>1), \] again after adjusting finitely many initial terms if necessary. \end{remark}

\end{document}